\documentclass[reqno,a4paper,11pt]{amsart}
\usepackage{amssymb,a4}
\usepackage{enumerate}
\usepackage{hyperref,color,url}
\usepackage{bbm} 
\usepackage{tikz}
\usetikzlibrary{matrix,arrows}

\newtheorem{theorem}{Theorem}
\newtheorem{lemma}[theorem]{Lemma}
\newtheorem{proposition}[theorem]{Proposition}
\newtheorem{corollary}[theorem]{Corollary}

\theoremstyle{definition}

\newtheorem{example}[theorem]{Example}

\newcommand\N{\mathbb N}

\newcommand\Q{\mathbb Q}
\newcommand\R{\mathbb R}
\newcommand\C{\mathbb C}

\newcommand{\ep}{\varepsilon}

\newcommand{\ta}{\tau}

\newcommand{\de}{\delta}

\newcommand{\io}{\iota}
\newcommand{\la}{\lambda}
\newcommand{\si}{\sigma}

\newcommand{\Ph}{\Phi}

\newcommand\RR{\R}

\renewcommand{\emptyset}{\varnothing}
\newcommand{\x}{\bar X}
\newcommand{\cx}{[\x]} 
\newcommand{\cH}{\mathcal H}

\newcommand{\pos}{\R_{>0}}

\newcommand{\rxt}{\R\cx^{t\times t}}
\newcommand{\srxt}{\sym\rxt}
\newcommand{\rrxt}{\RR\cx^{t\times t}}
\newcommand{\srrxt}{\sym\rrxt}
\newcommand{\Rtt}{\RR^{t\times t}}

\DeclareMathOperator{\tr}{tr}
\DeclareMathOperator{\sym}{Sym}

\renewcommand\i{\mathbbm i} 

\newcommand\cA{\mathcal A}
\newcommand\cR{\mathcal R}

\title[Pure states and matrix polynomials]
{Pure states, positive matrix polynomials\\and sums of hermitian squares}

\author[Igor Klep]{Igor Klep}
\address{Igor Klep, Univerza v Mariboru, Fakulteta za naravoslovje in matematiko,
Koro\v ska 160, Maribor, Slov\'enie, and Univerza v Ljubljani, Fakulteta za matematiko
in fiziko, 
Jadranska 21, 1111 Ljubljana, Slov\'enie}
\email{igor.klep@fmf.uni-lj.si}
\thanks{Both authors were supported by the French-Slovene partnership project Proteus 20208ZM}
\author[Markus Schweighofer]{Markus Schweighofer}
\address{Markus Schweighofer, Universit\'e de Rennes 1,
Laboratoire de Math\'ematiques,
Campus de Beaulieu,
35042 Rennes cedex,
France}
\email{markus.schweighofer@univ-rennes1.fr}
\date{September 29, 2009}

\begin{document}
\begin{abstract}
Let $M$ be an archimedean quadratic module of real $t\times t$ matrix polynomials in $n$ variables,
and let $S\subseteq\R^n$ be the set of all points where each element of $M$ is positive semidefinite.
Our key finding is a natural bijection between the set of pure states of $M$ and
$S\times\mathbb P^{t-1}(\R)$. This leads us to conceptual proofs of positivity certificates for
matrix polynomials, including the recent seminal result of Hol and Scherer: 
If a symmetric matrix polynomial is positive definite on $S$, then it belongs to $M$.
We also discuss what happens for nonsymmetric matrix polynomials or
in the absence of the archimedean assumption, and review some of the related classical results.
The methods employed are both algebraic and functional analytic.
\end{abstract}
\maketitle

\section{Introduction}

We write $\N:=\{1,2,\dots\}$, $\Q$, $\R$ and $\C$ for the sets of natural, rational, real and complex numbers, respectively. The complex numbers $\C$ always come equipped with the complex-conjugation
involution. For any matrix $A$ over a ring with involution $\cA$, we denote by $A^*$ its
conjugate transpose. If $A$ is a real matrix, $A^*$ is simply its transpose. Let $\sym\cA^{t\times t}:=
\{A\in\cA^{t\times t}\mid A=A^*\}$ be the set of all \emph{symmetric} $t\times t$ matrices.
Examples of these include \emph{hermitian squares}, i.e.,
elements of the form $A^*A$ for some $A\in\cA^{t\times t}$.

Recall that a matrix $A\in\Rtt$ is called
\textit{positive semidefinite} if it is
symmetric and $\langle Av,v\rangle=v^*Av\ge 0$ for all vectors $v\in\RR^t$, 
$A$
is \textit{positive definite} if it is positive semidefinite and invertible,
and is
called \textit{negative semidefinite} if $-A$ is positive semidefinite.
For matrices $A$ and $B$ of the same size, we write $A\preceq B$
(respectively $A\prec B$) to express that $B-A$ is positive semidefinite
(respectively positive definite).
Geometrically, $A\in\sym\Rtt$ is positive semidefinite if
and only if all of its eigenvalues are nonnegative, $A$ is positive
definite if and only if all of its eigenvalues are positive, and
$A$ is not negative semidefinite if and only if one 
of its eigenvalues is positive. The following algebraic characterizations are easy to prove:

\begin{proposition}\label{prop:0}
Let $A\in\sym\Rtt$.
\begin{enumerate}[\rm(a)]
\item\label{it:0psd}
$A\succeq 0$ if and only if $A$ is a sum of hermitian squares in $\Rtt$;
\item\label{it:0nnsd}
$A\not\preceq0$ if and only if there exist $B_i,C_j\in\Rtt$ such that
$$\sum_iB_i^*AB_i=1+\sum_jC_j^*C_j.$$
\end{enumerate}
\end{proposition}

The main goal of this article is to explain how this proposition extends to \emph{matrix polynomials}, i.e., elements of the ring $\RR\cx^{t\times t}$ where $\RR\cx$ is the ring of polynomials in $n$
variables $\bar X=(X_1,\dots,X_n)$ with coefficients from $\RR$. Note that in $\RR^{t\times t}$,
every sum of hermitian squares is of course a hermitian square. The reason why we speak of \emph{sums}
of hermitian squares in Proposition \ref{prop:0} is that this is no longer true in
$\RR\cx^{t\times t}$. Note however, that $A\in\RR\cx^{t\times t}$ is a sum of hermitian squares in
$\RR\cx^{t\times t}$
if and only if there is $u\in\N$ and $B\in\RR\cx^{u\times t}$ such that $A=B^*B$.

Let $\cA$ be a ring with \emph{involution} $a\mapsto a^*$ (i.e., $(a+b)^*=a^*+b^*$, $(ab)^*=b^*a^*$ and $a^{**}=a$ for $a,b\in\cA$) and set $\sym\cA:=\{a\in\cA\mid a=a^*\}$. A subset $M\subseteq\sym\cA$ is called a \emph{quadratic module} in $\cA$ if
$$1\in M,\quad M+M\subseteq M\quad\text{and}\quad a^*Ma\subseteq M\text{\ for all $a\in\cA$.}$$
To every $G\subseteq\srxt$, we associate the set
$$S_G:=\{x\in\RR^n\mid\forall g\in G:\, g(x)\succeq 0\}$$
and the quadratic module $M_G$ generated by $G$ in $\rrxt$. That is,
$$M_G=\left\{\sum_{i=1}^Np_i^*g_ip_i\mid N\in\N,g_i\in\{1\}\cup G,p_i\in\rrxt\right\}.$$
In particular, $M_\emptyset$ is the set of all sums of hermitian squares in $\rrxt$.

Given a matrix polynomial $f\in\sym\rrxt$ and $S\subseteq\RR^n$, we write $f\succeq 0$ on $S$ if
for all $x\in S$, $f(x)\succeq 0$. Likewise we use $f\succ 0$, $f\not\preceq 0$.
With this notation, $f\in M_G$ implies $f\succeq0$ on $S_G$.

In the sequel, we investigate how $M_G$ can be used to describe matrix polynomials $f\in\srxt$
with $f\succ0$, $f\succeq0$ or $f\not\preceq0$ on $S_G$. In Section \ref{sec:global}, the case
$G=\emptyset$ is considered; classical results on globally positive semidefinite matrix polynomials
in one or more variables are reviewed, 
and then we turn to nowhere negative semidefiniteness.
We give a sum of hermitian squares representation with denominators
in the one variable case
and prove mostly negative results for the case of matrix polynomials in 
several variables.

Our main results are presented in Section \ref{sec:arch} which is devoted to the case of \emph{compact}
$S_G$. Actually we work under the slightly stronger assumption that the quadratic module $M_G$ is
\emph{archimedean} (which can be enforced by possibly enlarging $G$ without changing $S_G$). Under
this assumption we describe all pure states on $\sym\rrxt$ (extremal linear forms positive with respect
to $M_G$) as being of the form $p\mapsto\langle p(x)v,v\rangle$ for some $x\in S_G$ and $v\in\R^t$.
From this we deduce certificates, in the spirit of Proposition
\ref{prop:0}, for matrix polynomials being nowhere negative semidefinite on $S_G$ or
positive semidefinite on $S_G$ in the spirit of Proposition
\ref{prop:0}. The latter was originally proved by Hol and Scherer \cite{hs}
with entirely different methods.

For a study of positivity of matrix polynomials in \emph{noncommuting} variables, we refer the reader
to Helton and McCullough, see e.g.~\cite{hm}. Burgdorf, Scheiderer and the second author \cite{bss}
investigate pure states and their role in commutative algebra.

\section{Globally positive matrix polynomials}\label{sec:global}

There are various notions of positivity for matrices. Like in Proposition \ref{prop:0}, we
consider positivity of the smallest and largest eigenvalue, respectively.

\subsection{Globally positive semidefinite matrix polynomials}

By Gau\ss ' theorem, every nonnegative univariate real polynomial is a
sum of two squares of real polynomials. The extension to univariate
real matrix polynomials was first given by Jakubovi\v c \cite{ja}
and is in a different form commonly known as the Kalman-Jakubovi\v c-Popov lemma
\cite{aip}.
It is one of the vast number of matrix factorization results obtained
and used in operator and control theory \cite{gks,glr,rr}.
We refer the reader to \cite{aip} for a nice algorithmic proof; see also \cite{dj1}.

\begin{theorem}[Jakubovi\v c]\label{thm:kyp}
For $f\in\sym\RR[Z]^{t\times t}$, the following are equivalent:
\begin{enumerate}[\rm(i)]
\item $f\succeq 0$ on $\RR$;
\item $f$ is a sum of two hermitian squares in $\RR[Z]^{t\times t}$.
\end{enumerate}
\end{theorem}

Note that (ii) is equivalent to $f=g^*g$ for some $g\in\RR[Z]^{2t\times t}$, or
$f=\sum_{i=1}^{2t}v_iv_i^*$ for some $v_i\in\RR[Z]^t$ (cf.~also \cite{clr,frs}).

The multivariate version of Gau\ss ' theorem is Artin's solution
to Hilbert's 17th problem \cite{ma,pd}: a nonnegative
multivariate real polynomial is a
sum of squares of real \emph{rational} functions.
A multivariate version of Jakubovi\v c's theorem (and at the same
time the matrix version of Artin's theorem) was obtained by
Gondard and Ribenboim \cite{gr} in 1974 and reproved several times, 
e.g.~\cite{dj2,ps,hin}.

\begin{theorem}[Gondard \& Ribenboim]\label{thm:h17}
For $f\in\srrxt$, the following are equivalent:
\begin{enumerate}[\rm(i)]
\item $f\succeq 0$ on $\RR^n$;
\item $p^2f$ is a sum of
hermitian squares in $\rrxt$ for some nonzero $p\in\RR\cx$.
\end{enumerate}
\end{theorem}

\begin{proof} From (ii) it follows that $f\succeq 0$ on $\{x\in\R^n\mid p(x)\neq0\}$ and hence
(i). Conversely, suppose that (i) holds. By diagonalization of quadratic forms over a field, there
exists an invertible matrix $g\in\RR(\x)^{t\times t}$ and a diagonal matrix $d\in\RR(\x)^{t\times t}$ such that $f=g^*dg$. By (i), $d$ is positive semidefinite where defined. By Artin's solution to 
Hilbert's 17th problem, we find a nonzero $p\in\RR\cx$ such that $p^2d$ is a sum of (hermitian)
squares in $\rrxt$. Without loss of generality, we can assume that $p^2$ also clears the denominators
in $g$.
\end{proof}

In the literature cited above, one can find refinements of (ii) at the expense of more complicated
proofs, e.g.~(ii')
$p^2f$ is a sum of squares in the commutative ring $\RR[\x,f]\subseteq\rrxt$
for some nonzero $p\in\RR\cx$. Also, Gondard and Ribenboim \cite{gr} prove a bound on the number of
hermitian squares needed.

\subsection{Nowhere negative semidefinite matrix polynomials}

We now turn to symmetric nonnegative matrix polynomials which are not negative semidefinite globally, i.e., whose evaluations
all have at least one positive eigenvalue.
We start by giving an analog of Proposition \ref{prop:0}\eqref{it:0nnsd}
for \emph{univariate} matrix polynomials. Though this is the perfect counterpart to the well known Theorem \ref{thm:kyp}, it is to the
best of our knowledge an entirely new result.

\begin{theorem}\label{thm:nnsd1var}
For $f\in\sym\RR[Z]^{t\times t}$, the following are equivalent:
\begin{enumerate}[\rm(i)]
\item $f\not\preceq 0$ on $\RR$;
\item 
 there exist $p_i\in\RR[Z]^{t\times t}$ such that
 $
 \sum_ip_i^*fp_i- 1 $ is a sum of hermitian squares.
\end{enumerate}
\end{theorem}

\begin{proof}
It is clear that (ii) $\Rightarrow$ (i), cf.~Proposition \ref{prop:0}\eqref{it:0nnsd}.

To prove the converse, suppose first that $f$ is diagonal, 
say $f=\left[\begin{smallmatrix} f_1 \\ & \ddots \\ & & f_t
\end{smallmatrix}\right]$. By assumption (i), $S_{\{-f_1,\ldots,-f_t\}}=\emptyset$ is compact.
Since we are in the univariate case, this implies that the quadratic
module $M_{\{-f_1,\ldots,-f_t\}}\subseteq\RR[Z]$ contains 
all polynomials positive on $S_{\{-f_1,\ldots,-f_t\}}$ \cite[Theorem 6.3.8]{pd}.
In particular, with $f_0:=-1$, there are $g_{ij}\in\RR[Z]$ satisfying 
\begin{equation}\label{eq:prel1}
-1=  \sum_{i=0}^t (-f_i)\sum_j g_{ij}^2.
\end{equation}

Observe that for each $i$, 
\begin{equation}\label{eq:prel2}
f_i=\sum_{k=1}^t E_{ik}^* fE_{ik} \in M_{\{f\}}
,
\end{equation}
where $E_{jk}$ are the $t\times t$ matrix units.
Thus \eqref{eq:prel1} implies
\begin{equation}\label{eq:prel3}
\sum_{i,k=1}^t \sum_{j} (E_{ik} g_{ij})^*  f  (E_{ik}  g_{ij}) -1 =
\sum_{i=1}^t \sum_j g_{ij}^* f_i g_{ij} -1 = 
\sum_j g_{0j}^2 
\end{equation}
is a sum of hermitian squares.

\def\rztt{\R[Z]^{t\times t}}
Now suppose $f$ is not necessarily diagonal. By the version of the 
LDU decomposition for matrix polynomials given in \cite[Proposition 8]{sm2},
there exist diagonal matrices $D_\ell\in\rztt$, and matrices $C_{\ell}\in\rztt$, 
$\ell=1,\ldots,m$, satisfying
\begin{enumerate}[\rm (a)]
\item
$D_\ell=C_{\ell}^* (-f) C_{\ell}$,
\item
for each $x\in\R$, $-f(x)\succeq0$ if and only if for all $\ell$, $D_\ell(x)\succeq0$.
\end{enumerate}

From (b) it follows that the diagonal matrix $\left[\begin{smallmatrix} -D_1 \\ & \ddots \\ & & -D_m
\end{smallmatrix}\right]$ is nowhere negative semidefinite. If $D_\ell$ is the diagonal
matrix with entries $d_{j,\ell}$, $j=1,\ldots, t$, then again by 
\cite[Theorem 6.3.8]{pd} we
deduce $-1\in M_{\{d_{j,\ell}\mid j=1,\ldots,t,\, \ell=1,\ldots,m\}}$.
Like in \eqref{eq:prel2} we have $d_{j,\ell}\in
 M_{\{D_\ell\}}\subseteq M_{\{D_1,\ldots,D_m\}}$ for all $j,\ell$. 
 Since $M_{\{D_1,\ldots,D_m\}}\subseteq
 M_{\{-f\}}$ by (a), we conclude as in \eqref{eq:prel3} that $-1\in M_{\{-f\}}$.
\end{proof}

Unlike in the univariate case or 
for Proposition \ref{prop:0}\eqref{it:0psd} where a satisfactory
statement on the level of \emph{multivariate} 
matrix polynomials has been given
in the previous subsection (see Theorem \ref{thm:h17}),
there does not seem to exist a straightforward extension of 
Proposition \ref{prop:0}\eqref{it:0nnsd} to the multivariate case.

\begin{example}\label{ex:22}
Consider $f\in\sym\RR\cx^{2\times 2}$.
We have $f\not\preceq0$ on $\RR^n$ if and only if $\tr f(x)>0$ or $\det f(x)<0$
for all $x\in\RR^n$.
\begin{enumerate}[(a)]
\item\label{part:22diag} Let $f$ be diagonal.
Then $f\not\preceq 0$ on $\RR^n$ if and only if there exist $p_i\in\RR\cx^{2\times 2}$
such that $\sum_ip_i^*fp_i\in 1+M_\emptyset$. Indeed, the implication $(\Leftarrow)$
is easy (cf.~Proposition~\ref{prop:0}). For the converse implication, suppose
$f=\begin{bmatrix}a&0\\0&c\end{bmatrix}\not\preceq0$ on $\RR^n$. Then
$a>0$ on $S_{\{-c\}}$ and therefore $p^2a=1+\si- \ta c$ for some $p\in\RR\cx$ and sums of squares
$\si,\ta\in\RR\cx$ by Krivine's Positivstellensatz (see, e.g.~\cite[Chapter 2]{ma} or
\cite[Section 4.2]{pd}).
Obviously, there exist diagonal $h_i\in\RR\cx^{2\times 2}$ such that
$\sum_ih_i^*fh_i=\begin{bmatrix}p^2a&0\\0&\ta c\end{bmatrix}$. Now
$$\sum_ih_i^*fh_i+\sum_i\begin{bmatrix}0&1\\1&0\end{bmatrix}^*h_i^*fh_i
\begin{bmatrix}0&1\\1&0\end{bmatrix}=\begin{bmatrix}1+\si&0\\0&1+\si\end{bmatrix}
\in 1+M_\emptyset.
$$
\item If $f=\begin{bmatrix}a&b\\b&c\end{bmatrix}$, then
\begin{equation}\label{eq:22}
\begin{bmatrix}
1&-b\\0&a
\end{bmatrix}^*
f
\begin{bmatrix}
1&-b\\0&a
\end{bmatrix}
+
\begin{bmatrix}
0&1\\-c&b
\end{bmatrix}^*
f
\begin{bmatrix}
0&1\\-c&b
\end{bmatrix}
=\begin{bmatrix}
 \tr f&0\\
 0&(\tr f)(\det f)
 \end{bmatrix}.
\end{equation}
Unable to settle the general case, we assume $\tr f(x)\neq0$ for all
$x\in\RR^n$. Then the diagonal matrix on the right
hand side of \eqref{eq:22} is nowhere negative semidefinite on $\RR^n$. By part
\eqref{part:22diag} above, we obtain $h_i\in\RR\cx^{2\times 2}$ such that
$\sum_ih_i^*fh_i\in1+M_\emptyset$.
\end{enumerate}
\end{example}

\begin{example}\label{ex:33}
The diagonal matrix
$$f:=\begin{bmatrix}X_1&0&0\\0&X_2&0\\0&0&X_1X_2+1\end{bmatrix}\in
\sym\RR\cx^{3\times 3}$$
satisfies $f\not\preceq0$ on $\R^2$ and yet there do not exist
$p_i\in\RR\cx^{2\times 2}$ with $\sum_ip_i^*fp_i\in1+M_\emptyset$.
This example is inspired by \cite[Example 7.3.2(i)]{ma} which is a modification
of the Jacobi-Prestel example \cite[Example 6.3.1]{pd}.

By way of contradiction, assume $\sum_ip_i^*fp_i=1+q$ where $q\in M_\emptyset$.
Extracting the top left entry on both sides of this equation, we get sums of squares
$\si_1,\si_2,\si_3,\ta\in\RR[X_1,X_2]$ with
$$\si_1X_1+\si_2X_2+\si_3(X_1X_2+1)=1+\ta.$$
In particular, $-1$ lies in the quadratic module generated in $\RR[X_1,X_2]$
by $-X_1$, $-X_2$ and $-X_1X_2-1$, contradicting the existence
of a semi\-ordering in $\RR[X_1,X_2]$ containing these three polynomials
(cf.~\cite[Example 7.3.1]{ma}).
\end{example}

We will see that a sum of hermitian squares representation with denominators
(and weights) \emph{does} exist for matrix polynomials nonnegative
on a compact set with archimedean corresponding quadratic module (see
Subsection \ref{subsec:archNN}).
This seems to mimic the situation for polynomials in noncommutative
variables studied e.g.~in \cite{ks}, where a \textit{Nirgendsnegativsemidefinitheitsstellensatz} describing nonnegativity on a bounded set has been given,
while the global case is still an open problem; see
\cite[Open problem 3.2]{ks} for a precise formulation.

\section{Archimedean quadratic modules of matrix polynomials}\label{sec:arch}

$C^*$-algebras $\cA$ enjoy the following boundedness property: for all $a\in\cA$ there is an $N\in\N$
such that $N-a^*a$ is a sum of hermitian squares (actually a hermitian square). In this section, we
try to mimic this boundedness property in an algebraic context for other rings with involution.
In the rigid context of (matrix) polynomials, sums of hermitian squares have of course
to be replaced by a general quadratic module.

\subsection{Archimedean quadratic modules}

A quadratic module $M$ of a ring with involution $\cA$ is said to be \textit{archimedean} if 
\begin{equation}\label{eq:arch}
\forall a\in\cA \ \exists N\in\N:\; N-a^*a \in M.
\end{equation}
To a quadratic module $M\subseteq\sym\cA$ we associate its 
\textit{ring of bounded elements} 
$$
H_M(\mathcal A):=\{a\in\cA\mid  \exists N\in\N:\; N-a^*a \in M\}.
$$
A quadratic module $M\subseteq\sym\cA$ is thus archimedean if
and only if $H_M(\cA)=\cA$. 
The name \emph{ring} of bounded elements
is justified by the following proposition originally
due to Vidav \cite{vi}; see also \cite{ci} for a more
accessible reference:

\begin{proposition}[Vidav]\label{prop:vidav}
Let $\cA$ be a ring with involution, $\frac12\in\cA$ and
$M\subseteq\sym\cA$ a quadratic module. Then $H_M(\cA)$ is a subring of $\cA$ and is closed under the involution.
\end{proposition}

In case $\cA$ is an $\R$-algebra, it suffices to check 
the archimedean condition \eqref{eq:arch} on a set of algebra generators.

\begin{lemma}\label{lem:arch}
A quadratic module $M\subseteq\srxt$ is archimedean if and only if
there exists $N\in\N$ with $N-\sum_iX_i^2\in M$.
\end{lemma}

\begin{proof}
The ``only if'' direction is obvious. For the converse, observe
that $\rxt$ is generated as an $\R$-algebra by $\x$ and
the matrix units $E_{ij}$, $i,j=1,\ldots,t$. By assumption,
$$
N-X_j^2= (N-\sum_iX_i^2 ) + \sum_{i\neq j}X_i^2\in M,
$$
so $X_j\in H_M(\rrxt)$ for every $j$. 
On the other hand, $E_{ij}^*E_{ij}= E_{jj}$ and
thus 
$$1-E_{ij}^*E_{ij}=\sum_{k\neq j} E_{kk}^*E_{kk} \in M.$$
Hence by Proposition \ref{prop:vidav}, $H_M(\rxt)=\rxt$ so $M$ is archimedean.
\end{proof}

\subsection{Pure states}

In functional analysis, the concept of pure states is well-established, see 
e.g.~\cite[Sections 1.6 and 1.7]{ar} for the classical application to $C^*$-algebras and their 
representations. Here we adopt these ideas to matrix polynomials.

Let $G\subseteq\srxt$. A linear form $L:\srxt\to\RR$ is called a \emph{state}
on $(\srxt,M_G)$ if $L(M_G)\subseteq\RR_{\ge0}$ and $L(1)=1$. A state $L$ is called \emph{pure} if it is
an extreme point of the convex set of all states, i.e., it is not a proper convex combination of two
states other than $L$.

We now come to the central result of this article. It is a matrix polynomial version of the well-known
theorem stating that for every pure state $L$ on a $C^*$-algebra $\cA$ there exists
a unit vector $v$ in a Hilbert space $\mathcal H$ and an
irreducible $*$-representation $\pi:\cA\to\mathcal B(\mathcal H)$ such that
$L(a)=\langle\pi(a)v,v\rangle$ for all $a\in\cA$ (see e.g.~\cite[Theorem 1.6.6]{ar}).

\begin{theorem}\label{thm:pure}
Suppose $G\subseteq\srxt$ and $M_G$ is archimedean.
For each pure state $L$ on $(\srxt,M_G)$, there exists $x\in S_G$ and a unit vector
$v\in\R^t$ such that
$$
L(p)=\langle p(x)v,v\rangle\qquad\text{for all $p\in\srxt$.}$$
\end{theorem}

\begin{proof}
We extend $L$ to $\C\cx^{t\times t}$ by setting
\begin{equation}\label{ext}
L(p+\i q)=\frac12(L(p+p^*)+\i L(q+q^*))
\end{equation}
for $p,q\in\rxt$. This is the unique $\C$-linear extension of $L$ satisfying $L(f^*)=L(f)^*$ for all
$f\in\C\cx^{t\times t}$. 
Let $$M_G^\C:=
\left\{\sum_{j=1}^Np_j^*g_jp_j\mid N\in\N,g_j\in\{1\}\cup G,p_j\in\C\cx^{t\times t}\right\}$$
be the quadratic module generated by $G$ in $\C\cx^{t\times t}$. Then $L$ is
nonnegative on $M_G^\C$. Indeed, given $f=(p+\i q)^*g(p+\i q)$ with
$p,q\in\rrxt$ and $g\in\{1\}\cup G$, we have
$f= (p^*gp+q^*gq)+\i(p^*gq-q^*gp)$. Applying the definition \eqref{ext} of $L$, we obtain
$L(f)=L(p^*gp + q^*gq)\in L(M_G)\subseteq\R_{\ge0}$,
as desired.
For later use let us observe that $M_G^\C$ is archimedean: write $f\in\C\cx^{t\times t}$ as $f=f_1+\i f_2$ with $f_j\in\rrxt$. Then $f_j\in H(M_G)\subseteq
H(M_G^\C)$ and $\i\in H(M_G^\C)$. Hence Proposition \ref{prop:vidav} implies
$H(M_G^\C)=\C\cx^{t\times t}$.

By the Cauchy-Schwarz inequality for semi-scalar products,
\begin{equation}\label{jdef}
J:=\{f\in\C\cx^{t\times t}\mid L(f^*f)=0\}
\end{equation}
is a linear subspace of $\C\cx^{t\times t}$.
Similarly, we see that
\begin{equation}\label{gns}
\langle\overline p,\overline q\rangle:=L(q^\ast p)
\end{equation}
defines a scalar product on $\C\cx^{t\times t}/J$, where
$\overline p:=p+J$ denotes the residue class of $p\in\C\cx^{t\times t}$ modulo $J$.
Let $\cH$ denote the completion of $\C\cx^{t\times t}/J$ with respect to this scalar product.
Note that $\cH\neq\{0\}$ since $1\not\in J$.

We proceed to show $J$ is a left ideal of $\C\cx^{t\times t}$. Let $f\in\C\cx^{t\times t}$.
Since $M_G^\C$ is archimedean, there is some $N\in\N$ with $N-f^*f \in M_G^\C$.
Hence for all $p\in\C\cx^{t\times t}$, we have
\begin{equation}\label{key}
0 \leq L(p^*(N-f^*f)p)\le NL(p^* p).
\end{equation}
This shows that $L(p^*f^*fp)=0$ for all $p\in J$, i.e., $fp\in J$.

Because $J$ is a left ideal, the map
\begin{equation}\label{pidef}
\pi:\C\cx^{t\times t}\to\mathcal B(\cH),\;f\mapsto(\overline p\mapsto\overline{fp})
\end{equation}
is well-defined. Here $\overline p\mapsto\overline{fp}$ stands for the unique
bounded linear extension to $\cH$ of the left multiplication with $\overline f$ on
$\C\cx^{t\times t}/J$, which is well-defined by \eqref{key}.
Using the definition \eqref{gns} of the scalar product, it is easy to see
that $\pi$ is a homomorphism of rings with involution, i.e., a $*$-representation of
$\C\cx^{t\times t}$ on $\cH$. Setting $w:=\overline 1\in\cH$, we observe that
\begin{equation}\label{pil}
L(f)=\langle\pi(f)w,w\rangle
\end{equation}
for all $f\in\C\cx^{t\times t}$.

We claim that the commutant $\pi(\C\cx^{t\times t})'$ of the image of $\pi$ in $\mathcal B(\cH)$
is $\C$. To see this, we take an arbitrary operator $T\in\pi(\C\cx^{t\times t})'$.
Since the commutant is closed under the involution and
$T=\frac{T+T^*}2+\i\frac{T-T^*}{2\i}$, we are reduced to the case $T=T^*$. By the spectral theorem,
$T$ decomposes into projections belonging to $\{T\}''\subseteq\pi(\C\cx^{t\times t})'$. So we can
even assume $T$ is a projection. By way of contradiction, assume $T\neq 0$ and $T\neq 1$. Since
$T\in\pi(\C\cx^{t\times t})'$ and $w$ is a cyclic vector for $\pi$ by construction, it follows
that $Tw\neq 0$ and $(1-T)w\neq 0$. This allows us to define states $L_i$ on $(\srxt,M_G)$ by
$$L_1(f)=\frac{\langle\pi(f)Tw,Tw\rangle}{\|Tw\|^2}
  \qquad\text{and}\qquad
  L_2(f)=\frac{\langle\pi(f)(1-T)w,(1-T)w\rangle}{\|(1-T)w\|^2}$$
for all $f\in\srxt$. One checks that $L$ is a convex combination of $L_1$ and $L_2$. The state $L$
being pure, we obtain $L=L_i$. By \eqref{pil}, this implies
$$\langle\pi(f)w,\la w\rangle=\la\langle\pi(f)w,w\rangle=\langle\pi(f)Tw,Tw\rangle=\langle T\pi(f)w,Tw\rangle=
\langle \pi(f)w,Tw\rangle$$
for all $f\in\C\cx^{t\times t}$, where $\la:=\|Tw\|^2$. In particular, $Tw=\la w$ since $w$
is a cyclic vector for $\pi$. This implies $\la\in\{0,1\}$ since $T$ is a projection, a contradiction.

By \cite[Theorem 3.1]{la}, $\ker\pi=I^{t\times t}$ for an ideal $I$ of $\C\cx$. Since $\ker\pi$ is
closed under the involution, $I$ is closed under complex conjugation. Moreover
$\C\cx/I$ is contained in the center of
$(\C\cx/I)^{t\times t}=\C\cx^{t\times t}/\ker\pi\cong\pi(\C\cx^{t\times t})$
which is $\C$ by the above. Hence $\C\cx/I=\C$ and $\pi(\C\cx^{t\times t})\cong\C^{t\times t}$
as a $C^*$-algebra. In particular, there
exists $x\in\C^n$ such that $I=\{p\in\C\cx\mid p(x)=0\}$.
Actually $x\in\R^n$ since $I=I^*$. Also, $\cH=\pi(\C\cx^{t\times t})w$ is finite-dimensional.

Next we claim that $\pi$ is an irreducible $*$-representation. Indeed, suppose $U$ is a
linear subspace of $\cH$ invariant under every $\pi(f)$ for $f\in\C\cx^{t\times t}$.
Let $P:\cH\to U$ denote the orthogonal projection. It suffices to show that
$P\in\pi(\C\cx^{t\times t})'=\C$, i.e., $\pi(f)P=P\pi(f)$ for each
$f\in\C\cx^{t\times t}$. By the standard trick, we reduce to the case $f=f^*$.
But then $$\pi(f)P=P\pi(f)P=(P\pi(f)P)^*=(\pi(f)P)^*=P\pi(f).$$

As indicated in the commutative diagram below, $\pi$ now induces an irreducible $*$-representation
$\bar\pi$ of $\C^{t\times t}$. This representation is unitarily equivalent to the identity
representation $\io$
\cite[Corollary 2 to Theorem 1.4.4]{ar}, i.e., there is a unitary map $\Ph:\cH\to\C^t$ making the
diagram below commute.
\vspace{-2.1em}
\begin{center}
\begin{tikzpicture}[description/.style={inner sep=2pt}] 
\matrix (m) [matrix of math nodes, row sep=1em, 
column sep=2em, text height=1.5ex, text depth=0.25ex] 
{\C\cx^{t\times t} & C\cx^{t\times t}/\ker\pi & \C\cx^{t\times t}/I^{t\times t} &
(\C\cx/I)^{t\times t} & \C^{t\times t}\\ 
\\
\\
&\mathcal B(\cH) & & \mathcal B(\C^t)\\
&\cH&&\C^t\\
}; 
\path[->,font=\scriptsize] 
(m-1-1) edge[bend left] node[below] {$f\mapsto f(x)$} (m-1-5)
(m-1-1) edge (m-1-2)
(m-1-2) edge[-,double distance=2pt] (m-1-3)
(m-1-3) edge[-,double distance=2pt] (m-1-4)
(m-1-4) edge[-,double distance=2pt] (m-1-5)
(m-1-1) edge node[left,below] {$\pi$} (m-4-2)
(m-1-5) edge[-,double distance=2pt] node[auto] {$\io$} (m-4-4)
(m-1-5) edge node[auto] {$\bar\pi$} (m-4-2)
(m-4-2) edge node[below] {$T\mapsto\Ph T\Ph^*$} (m-4-4)
(m-5-2) edge node[below] {$\Ph$} node[above]{$\scriptstyle\overline 1=w\mapsto u=\Ph(w)$} (m-5-4);
\end{tikzpicture} 
\end{center}
Let $u:=\Ph(w)\in\C^t$. For each $p\in\C\cx^{t\times t}$, we have
\begin{equation}\label{close}
\begin{split}
L(p)&=\langle\pi(p)w,w\rangle=\langle\Ph\pi(p)w,\Ph w\rangle=\langle\Ph\pi(p)\Ph^*u,u\rangle\\
&=\langle\Ph\bar\pi(p(x))\Ph^*u,u\rangle=\langle\io(p(x))u,u\rangle=\langle p(x)u,u\rangle.
\end{split}
\end{equation}
In particular, we get $\tr(Auu^*)=\langle Au,u\rangle=L(A)\in\R$ for all $A\in\R^{t\times t}$.
This implies $uu^*\in\R^{t\times t}$, that is, $uu^*$ is a real positive semidefinite rank one
matrix and can thus be factorized as $uu^*=vv^*$ for some $v\in\R^t$. We can now rewrite
\eqref{close} as
$$L(p)=\langle p(x)u,u\rangle=\tr(p(x)uu^*)=\tr(p(x)vv^*)=\langle p(x)v,v\rangle$$
for all $p\in\rxt$. Also note that $\langle v,v\rangle=L(1)=1$, i.e., $v$ is a unit vector.

It remains to show that $x\in S_G$. To show this, let $g\in G$ and $z\in\R^t$. Choose
$A\in\R^{t\times t}$ with $z=Av$. Then
$$\langle g(x)z,z\rangle=\langle g(x)Av,Av\rangle=\langle A^*g(x)Av,v\rangle=L(A^*gA)\ge 0$$ since
$A^*gA\in M_G$.
\end{proof}

\begin{proposition}\label{thm:ispure}
Let $G\subseteq\srxt$.
For each $x\in S_G$ and each unit vector $v\in\R^t$, the state $L$ on $(\srxt,M_G)$ defined by
$L(p)=\langle p(x)v,v\rangle$ is pure.
\end{proposition}

\begin{proof}
For convenience of notation, set $L_0:=L$.
Suppose there are states $L_1$ and $L_2$ on $(\srxt,M_G)$ such that $2L_0=L_1+L_2$. Extending each
$L_i$ to $\C\cx^{t\times t}$ as in \eqref{ext}, this equality still holds. Define linear subspaces
$J_i\subseteq\C\cx^{t\times t}$ as in \eqref{jdef} by $J_i:=\{f\in\C\cx^{t\times t}\mid L_i(f^*f)=0\}$.
Obviously
\begin{equation}\label{obeq}
J_1\cap J_2=J_0=\{f\in\C\cx^{t\times t}\mid f(x)v=0\}.
\end{equation}
In particular, $\C\cx^{t\times t}/J_0\cong\C^t$ as vector spaces. Hence $\cH_i:=\C\cx^{t\times t}/J_i$
is of dimension at most $t$. The GNS construction for $L_i$ yields a scalar product
$\langle.,.\rangle_i$ on $\cH_i$
defined as in \eqref{gns}
and a $*$-representation $\pi_i:\C\cx^{t\times t}\to\mathcal B(\cH_i)$
(cf.~the proof of the previous theorem).
Again by \cite[Theorem 3.1]{la}, there are ideals $I_i\subsetneq\C\cx$ such that
$\ker\pi_i=I_i^{t\times t}$ and therefore $\C\cx^{t\times t}/\ker\pi_i\cong(\C\cx/I_i)^{t\times t}$.
In particular, $$\dim\mathcal B(\cH_i)\ge\dim(\C\cx^{t\times t}/\ker\pi_i)\ge t^2\dim(\C\cx/I_i)\ge t^2$$ whence
$\dim\cH_i\ge t$ and therefore $\dim\cH_i=t$. Now by \eqref{obeq}, we have $J_0=J_1=J_2$.
Therefore we have three scalar products on $\cH:=\cH_0=\cH_1=\cH_2$, and we find
positive definite matrices $G_1,G_2\in\C^{t\times t}$ such that
$$L_i(q^*p)=\langle\overline p,\overline q\rangle_i=v^*q(x)^*G_ip(x)v$$ where $\overline p=p(x)v$ denotes the residue class of $p$ modulo $J_i$ (this is also true for $i=0$ with $G_0$ being the identity
matrix). Since
$$v^*C^*G_iABv=\langle\overline{AB},\overline C\rangle_i=
L_i(C^*AB)=\langle\overline B,\overline{A^*C}\rangle_i=v^*C^*AG_iBv$$
for all $A,B,C\in\C^{t\times t}$, it follows that $G_iA=AG_i$ for all $A\in\C^{t\times t}$, i.e.,
$G_i\in\C$. More precisely, $G_i=G_i\langle v,v\rangle=v^*G_iv=L_i(1)=1$. Thus $L=L_1=L_2$.
\end{proof}

For $0\neq v\in\R^t$, denote by $[v]$
the linear subspace spanned by $v$ seen as an element of the real projective
space $\mathbb P^{t-1}(\R)$ of dimension $t-1$.

\begin{corollary}\label{cor:proj}
Suppose $G\subseteq\srxt$ and $M_G$ is archimedean. We have a bijection
between $S_G\times\mathbb P^{t-1}(\R)$ and
the set of pure states on $(\srxt,M_G)$ well-defined by associating to each
$(x,[v])\in S_G\times\mathbb P^{t-1}(\R)$, $v$ a unit vector of $\R^t$,
the map $p\mapsto\langle p(x)v,v\rangle$.
\end{corollary}

\begin{proof}
By Proposition \ref{thm:ispure} and $\langle p(x)v,v\rangle=\langle p(x)(-v),-v\rangle$,
the map is well-defined. It is surjective by Theorem \ref{thm:pure}. To show that
it is injective, let $(x,[v]),(y,[w])\in S_G\times\mathbb P^{t-1}(\R)$ with unit vectors
$v,w\in\R^t$ satisfy
\begin{equation}\label{idem}
\langle p(x)v,v\rangle=\langle p(y)w,w\rangle
\end{equation}
for all $p\in\srxt$. Using \eqref{idem} with $p=E_{ij}+E_{ji}$ yields $vv^*=ww^*$. Then $[v]=[w]$ since 
$(v^*w)v=vv^*w=ww^*w=w$. Setting $p=X_i$ in \eqref{idem}, we get moreover $x_i=y_i$ whence $x=y$.
\end{proof}

In general, Theorem \ref{thm:pure} and Corollary \ref{cor:proj} fail badly for
nonarchimedean $M_G$.

\begin{example}
Take $G=\emptyset$, $t=1$ and $n=2$, i.e., consider pure states on $(\R[X,Y],M_\emptyset)$ where
$X$ and $Y$ are two variables and $M_\emptyset$ is the cone of sums of squares of polynomials.
We endow the (algebraic) dual $\R[X,Y]^\vee$ with the weak$^*$ topology and
consider the closed convex cone $M_\emptyset^\vee\subseteq\R[X,Y]^\vee$ of linear forms $L:\R[X,Y]\to\R$ with $L(M_\emptyset)\subseteq\R_{\ge 0}$. Note that each $0\neq L\in M_\emptyset^\vee$ becomes a
state on $(\R[X,Y],M_\emptyset)$ after multiplication with a positive scalar (for $L(1)=0$ implies
$L=0$ by a Cauchy-Schwarz argument).
Choose a polynomial $f$ with $f\ge 0$ on $\R^2$ and $f\notin M_\emptyset$, for instance the
Motzkin polynomial $f=X^2Y^4+X^4Y^2-3X^2Y^2+1$ \cite[Proposition 1.2.2]{ma}.
Since $M_\emptyset$
is closed in $\R[X,Y]$ with respect to the finest locally convex topology (see, e.g.,
\cite[Proposition 4.1.2(2)]{ma} together with \cite[Example 4.1.5]{ma}), the Hahn-Banach separation
theorem \cite[Theorem III.3.4]{ba} yields $L_0\in M_\emptyset^\vee$ with $L_0(f)<0$.
As explained above, we can assume that $L_0$ is a state on $(\R[X,Y],M_\emptyset)$. 

Fix a double sequence $(c_{ij})_{i,j\in\N_0}$ of $c_{ij}>0$ satisfying
$\sum_{i,j}c_{ij}L_0(X^{2i}Y^{2j})=1$.
Now set $$C:=\left\{L\in M_\emptyset^\vee\mid\sum_{i,j\in\N_0}c_{ij}L(X^{2i}Y^{2j})\le1\right\}.$$ Then $C$ is
weak$^*$ closed since
$C=\bigcap_{k\in\N_0}\{L\in M_\emptyset^\vee\mid\sum_{i,j=0}^kc_{ij}L(X^{2i}Y^{2j})\le1\}$
and compact (for if $L\in C$, then
$|L(X^{2i}Y^{2j})|\le\frac1{c_{ij}}$ for all $i,j\in\N_0$, and this implies
by a Cauchy-Schwarz argument similar a priori bounds for the values of $L$ on the other polynomials).
In addition,
both $C$ and $M_\emptyset^\vee\setminus C$ are obviously convex.
Hence $C$ is a \emph{cap} of $M_\emptyset^\vee$ containing
$L_0$ (see \cite[page 80]{ph}).

By the Krein-Milman theorem \cite[Theorem III.4.1]{ba}, there exists
an extreme point $L$ of $C$ such that $L(f)<0$. By Choquet theory \cite[Proposition 13.1]{ph},
$L$ lies on an extreme ray of $M_\emptyset^\vee$. After normalization, $L$ is a
pure state on $(\R[X,Y],M_\emptyset)$. Since $L(f)<0$ and $f\ge0$ on $\R^2$,
$L$ cannot be a point evaluation.
\end{example}

\subsection{Positive semidefinite matrix polynomials}\label{subsec:archPSD}

Now we are ready to give a version of Proposition \ref{prop:0}\eqref{it:0psd} for matrix polynomials
in the archimedean case, originally due to Hol and Scherer \cite[Corollary 1]{hs}.
Using the above classification of pure states, the proof reduces to an easy
separation argument. In contrast to this, the original proof of Hol and Scherer is more involved.

\begin{theorem}[Hol \& Scherer]\label{thm:pos}
Suppose $G\cup\{f\}\subseteq\srxt$ and $M_G$ is archimedean.
If $f\succ 0$ on $S_G$, then $f\in M_G$.
\end{theorem}

\begin{proof}
In the terminology of Barvinok \cite[Definition III.1.6]{ba}, Proposition
\ref{prop:vidav} together with the identity
$4s=(s+1)^2-(s-1)^2$ shows that
$M_G$ is an archimedean quadratic module if and only if $1$ is an
algebraic interior point of the convex cone $M_G\subseteq\srxt$.
Recall: $f$ is an \textit{algebraic interior point} of $M_G$ if for
every $p \in \srxt$ there exists $\ep >0$ with $f+\ep p\in M_G$.

Suppose $f\notin M_G$. We will find $x\in S_G$ such that $f(x)\not\succ0$. The existence of an algebraic
interior point of $M_G$ allows us to separate the convex
sets $M_G$ and $\R_{>0}f$ by the Eidelheit-Kakutani separation theorem \cite[Theorem III.1.7]{ba}.
More precisely, there exists a state $L$ on $(\srxt,M_G)$ with $L(f)\le0$. The set of all such states
is weak$^*$ compact by Tikhonov's theorem (cf.~the proof of Alaoglu's theorem
\cite[Theorem III.2.9]{ba}). Hence by the Krein-Milman theorem \cite[Theorem III.4.1]{ba}, $L$ can
be chosen to be pure.

By Theorem \ref{thm:pure}, there exists $x\in S_G$ and a unit vector $v\in\R^t$ such that
$L(p)=\langle p(x)v,v\rangle$ for all $p\in\srxt$. In particular, $\langle f(x)v,v\rangle=L(f)\le0$
as desired.
\end{proof}

\begin{corollary}
Suppose $G\subseteq\srxt$ and $M_G$ is archimedean.
For $f\in\srxt$, the following are equivalent:
\begin{enumerate}[\rm (i)]
\item $f\succeq 0$ on $S_G$;
\item $f+\ep\in M_G$ for all $\ep\in\pos$.
\end{enumerate}
\end{corollary}

For $t=1$, Theorem \ref{thm:pos} specializes to Putinar's Positivstellensatz \cite{pu}.
To avoid possible confusion, we use the letter $Q$ to denote quadratic modules
in commutative rings with trivial involution. For instance, if $G\subseteq\R\cx$ we denote
the quadratic module generated by $G$ in $\R\cx$ by $Q_G$, i.e.,
$$Q_G=\left\{\sum_{i=1}^Np_i^2g_i\mid N\in\N,g_i\in\{1\}\cup G,p_i\in\R\cx\right\}.$$
Note that $Q_G=M_G$ for $t=1$ but $Q_G\subsetneq M_G$ for $t>1$.

\begin{corollary}[Putinar]\label{cor:put}
Suppose $G\cup\{f\}\subseteq\R\cx$ and $Q_G$ is archimedean.
If $f>0$ on $S_G$, then $f\in Q_G$.
\end{corollary}

Clearly, if $Q_G$ is archimedean then $S_G$ is compact. The converse is false even for finite
$G\subseteq\R\cx$ as shown by the Jacobi-Prestel example \cite[Example 6.3.1]{pd}. Nevertheless,
there is an intimate connection between compactness and the archimedean property established by
Schm\"udgen \cite{sm}. To describe his result, we introduce the following notation:
Given a set $G=\{g_1,\dots,g_m\}\subseteq\R\cx$ of $m$ distinct polynomials, let
$\widehat G=\{g_1^{\de_1}\dotsm g_m^{\de_m}\mid0\ne\de\in\{0,1\}^m\}$ denote
the set of the $2^m-1$ nontrivial products of the $g_i$.

\begin{theorem}[Schm\"udgen]\label{prop:schmu}
Suppose $G\subseteq\R\cx$ is finite. Then $S_G$ is compact if and only if the $($multiplicative$)$
quadratic module $Q_{\widehat G}$ is archimedean.
\end{theorem}

As an important special case, we obtain that for a singleton $G=\{g\}\subseteq\R\cx$,
$S_G$ is compact if and only if $Q_G$ is archimedean. This continues to hold if $G$ has exactly
two elements \cite[Corollary 6.3.7]{pd}. For this and other nontrivial strengthenings of
Schm\"udgen's theorem due to Jacobi and Prestel we refer to \cite[Chapter 6]{pd}. These results allow
us to deduce that $M_G$ is archimedean in such cases by the following proposition.

\begin{proposition}\label{prop:qm}
If $G\subseteq\R\cx$, then $Q_G$ is archimedean if and only if $M_G$ is
archimedean.
\end{proposition}

\begin{proof}
To prove the nontrivial direction, suppose that $M_G$ is archimedean, i.e.,
$N-\sum_{i=1}^nX_i^2=\sum_jp_j^*p_jg_j$ for some $N\in\N$, $p_j\in\rxt$ and $g_j\in G\cup\{1\}$.
Since the trace of a hermitian square $p_j^*p_j$ is a sum of squares in $\R\cx$, it follows
that $N-\sum_{i=1}^nX_i^2=\frac 1t\sum_j\tr(p_j^*p_j)g_j\in Q_G$. Hence $Q_G$ is archimedean
by Lemma \ref{lem:arch}.
\end{proof}

There does not seem to exist a viable generalization of Schm\"udgen's theorem for general
finite $G\subseteq\srxt$. It does not make sense to
consider $\widehat G$ because products of positive semidefinite matrices are not symmetric in general,
let alone positive semidefinite. If $G$ is a singleton, we have $\widehat G=G$ but still $S_G$
compact does not imply $M_G$ archimedean.

\begin{example}
Let $f\in\sym\R\cx^{3\times 3}$ be the diagonal matrix from Example \ref{ex:33}.
Then $S_{\{-f\}}=\emptyset$ is compact but $M_{\{-f\}}$ is not archimedean. Otherwise Theorem
\ref{thm:pos} would imply $-1\in M_{\{-f\}}$ which is not the case as seen in Example \ref{ex:33}.
\end{example}

We now briefly turn to positivity of \emph{not necessarily symmetric} matrix polynomials.
For this we need the following classical lemma \cite[Section 6.3]{br}
(see also \cite[Theorem 5.3]{sw}).

\begin{lemma}[Brumfiel]\label{lem:brumfiel}
Let $\cR$ be a commutative $\Q$-algebra and $Q\subseteq\cR$ a quadratic module.
Then $H_Q(\cR)$ is integrally closed in $\cR$.
\end{lemma}

Let $G\subseteq\R\cx$ and $f\in\rxt$. The quadratic module generated by $G$ in the commutative ring
$\R[\x,f]$ endowed with the \emph{trivial} involution will be denoted by $Q_G^{f}$. Observe that
$Q_G^{f}\subseteq M_G$ if and only if $f=f^*$.

\begin{theorem}\label{thm:qgf}
Suppose $G\subseteq\R\cx$, $Q_G$ is archimedean and $f\in\rxt$.
If for all $x\in S_G$, all {\em real} eigenvalues of $f(x)$ are positive,
then $f\in Q_G^f$.
\end{theorem}

\begin{proof}
Let $q_f\in\R(\x)[Y]$ be the minimal polynomial of the matrix $f$. Note that
$q_f\in\R[\x,Y]$ by Gau\ss' lemma since $q_f$ divides the (monic) characteristic polynomial of $f$
by the Cayley-Hamilton theorem.

Now
$$
Y> 0 \quad\text{on}\quad \{(x,y)\in S_G\times\R\mid q_f(x,y)=0\}
=S_{G\cup\{q_f,-q_f\}}.
$$
We claim that $Q_{G\cup\{q_f,-q_f\}}=Q_G^Y+\R[\x,Y]q_f$ is archimedean, or
equivalently,
the quadratic module $Q_G^f$ is archimedean in $\R[\x,f]=\R[\x,Y]/(q_f)$. Indeed,
since $f$ is integral
over $\R[\x]$ and $H_{Q_G^f}(\R[\x,f])\supseteq H_{Q_G}(\R\cx)=\R\cx$ is integrally closed,
we have $H_{Q_G^f}(\R[\x,f])=\R[\x,f]$.

By Corollary \ref{cor:put}, $Y\in Q_{G\cup\{q_f,-q_f\}}$. Plugging in $f$ for $Y$ yields
$f\in Q_G^f$.
\end{proof}

\begin{corollary}
Suppose $G\subseteq\R\cx$ and $Q_G$ is archimedean.
For $f\in\rxt$, the following are equivalent:
\begin{enumerate}[\rm (i)]
\item for all $x\in S_G$, all {\em real} eigenvalues of $f(x)$ are nonnegative;
\item $f+\ep\in Q_G^f$ for all $\ep\in\pos$.
\end{enumerate}
\end{corollary}

\begin{proof}(i)$\Rightarrow$(ii) follows from Theorem \ref{thm:qgf}. For the converse, it suffices
to observe that for $p\in Q_G^p$, all real eigenvalues of $p(x)$ are nonnegative for all
$x\in S_G$. Indeed, suppose $p=\sum_jh_j(\x,p)^2g_j$ for finitely many
$g_j\in G\cup\{1\}$ and $h_j\in\R[\x,Y]$. Let $\la\in\R$, $0\neq v\in\R^t$ and $p(x)v=\la v$. Then
$$\la v=p(x)v=\sum_j h_j(x,p(x))^2g_j(x)v=\sum_j h_j(x,\la)^2g_j(x)v$$
whence
$\la=\sum_j h_j(x,\la)^2g_j(x)\ge0$.
\end{proof}

\subsection{Matrix polynomials not negative semidefinite}\label{subsec:archNN}

We conclude this article with an application of Theorem \ref{thm:pos} yielding a version of
Proposition \ref{prop:0}\eqref{it:0nnsd} for matrix polynomials.

\begin{corollary}[Matrizenpolynomnirgendsnegativsemidefinitheitsstellensatz]\label{cor:nn}
\quad\par\noindent
Suppose $G\subseteq\srxt$ and $M_G$ is archimedean.
For $f\in\srxt$, the following are equivalent:
\begin{enumerate}[\rm (i)]
\item $f\not\preceq 0$ on $S_G$;
\item there exist $p_i\in\rxt$ such that
$$
\sum_ip_i^*fp_i\in 1+M_G.
$$
\end{enumerate}
\end{corollary}

\begin{proof}(ii)$\Rightarrow$(i) is immediate from Proposition \ref{prop:0}\eqref{it:0nnsd}.
For the converse, note that $f\not\preceq 0$ on $S_G$ if and only if
$S_{G\cup\{-f\}}=\emptyset$. In this case, $-1\succ0$ on $S_{G\cup\{-f\}}$. Since $M_G$ and therefore
$M_{G\cup\{-f\}}$ is archimedean, Theorem \ref{thm:pos} implies that $-1\in M_{G\cup\{-f\}}$ which is
exactly what we need.
\end{proof}

\subsection*{Acknowledgments.} We thank Ronan Quarez for interesting 
discussions which led us to the discovery of Theorem 4.
We also appreciate the anonymous referee for his careful reading of the manuscript and his corrections.

\end{document}